\documentclass[11pt]{article}
\usepackage[numbers]{natbib}
\usepackage{graphicx,amsmath,epsfig,amssymb,amsfonts,amsthm,enumerate,verbatim,todonotes,hyperref,mathtools, dsfont}

\usepackage{tikz,tikz-cd}
\usetikzlibrary{matrix,arrows}

\newtheorem{theorem}{Theorem}[section]
\newtheorem{corollary}[theorem]{Corollary}
\newtheorem{lemma}[theorem]{Lemma}
\newtheorem{remark}[theorem]{Remark}
\newtheorem{proposition}[theorem]{Proposition}

\newtheorem{question}[theorem]{Question}
\newtheorem{claim}[theorem]{Claim}

\theoremstyle{definition}
\newtheorem{definition}[theorem]{Definition}

\newcommand{\restrict}{\mathord{\upharpoonright}}

\begin{document}

\title{Effective Localization number: building $k$-surviving degrees}
\author{Iv\'an Ongay-Valverde\\\emph{Department of Mathematics}\\\emph{University of Wisconsin--Madison}\\\emph{Email: ongay@math.wisc.edu}\vspace{.4cm}\\Noah Schweber\\\emph{Department of Mathematics}\\\emph{University of Wisconsin--Madison}\\\emph{Email: schweber@wisc.edu}
}

\date{\begin{tabular}{rl}First Draft:&May 15, 2017\\Current Draft:&\today\end{tabular}}
\maketitle
\begin{abstract} We introduce and study effective versions of the localization numbers introduced by Newelski and Roslanowski \cite{newelski1993ideal}. We show that proper hierarchies are produced, and that the corresponding highness notions are relatively weak, in that they can often be made computably traceable. We discuss connections with other better-understood effective cardinal characteristics.
\end{abstract}

\tableofcontents

\section{Introduction}

In this paper we continue the study of computability-theoretic cardinal characteristics of the continuum. Classically, a cardinal characteristic is a cardinal which measures how large a set of reals with a certain ``sufficiency" property must be: for example, the least size of a set of functions $\omega\rightarrow\omega$ such that every function $\omega\rightarrow\omega$ is dominated by some function in the set. It is possible that all reasonable cardinal characteristics are equal --- this would follow from the continuum hypothesis --- but a rich structure is revealed when we look at {\em consistent } separations: the  study of provable (weak) cardinal characteristic inequalities, of possible simultaneous separations of more than two characteristics at once, and of the interactions between cardinal characteristic inequalities and properties of forcing notions which preserve or induce them are important aspects of modern set theory.

As is often the case, the theory of cardinal characteristics has an ``effective" counterpart.  The explicit analogy was first drawn by Rupprecht \cite{rupprechtthesis},\footnote{An early effective analogue of a cardinal characteristic equality was provided by Terwijn and Zambella \cite{terwijn2001computational}, although they did not draw this connection explicitly.} who analyzed the characteristics occurring in Cichon's diagram, and was further studied by others including Brendle, Brooke-Taylor, Ng, and Nies \cite{bbtnn}. Given a relation $R\subseteq (\omega^\omega)^2$, a set-theoretic cardinal characteristic emerges when we ask when a real $s$ has the property that $cRs$ for every computable real $c$. Associated to this question is a corresponding ``highness" property, and Rupprecht showed that these highness properties are often of independent interest in computability theory. We think of this highness property as measuring the sense in which the relevant sufficiency property is hard to achieve: for example, it is harder to produce a set of functions which dominates every function $\omega\rightarrow\omega$ than it is to produce one which escapes every function $\omega\rightarrow\omega$, and the corresponding inequality on the computability-theoretic side is ``high implies hyperimmune." Some cardinal characteristics require some appropriate coding to effectivize, such as $cov(\mathcal{L})$ = the smallest number of null sets which cover $\mathbb{R}$, but such coding can be done in a natural way via effective notions of null/meager sets. 

In this article we continue the study of effective cardinal characteristics, focusing on a class of invariants introduced by Newelski and Roslanowski \cite{newelski1993ideal} and further studied by Geschke and Kojman \cite{geschkekojman}. These measure roughly how many copies of a closed set it takes to fill a given space --- for example, how small a set of $2$-branching trees $T\subseteq 3^{<\omega}$ can be while still having every element of $3^\omega$ be a branch through a tree in the set. These characteristics lend themselves to multiple computability-theoretic interpretations which we explore, especially in connection with computable traceability. We show that the various notions so resulting are reasonably distinct, and exhibit a mixture of strength and weakness: for example, the simplest effective localization notions are ``is not a path through a computable $k$-branching subtree of $(k+1)^{<\omega}$" ({\em $k$-surviving}) for $k>1$. We study these in Section $2$. Our main result in this section is that these notions form a strict hierarchy, and interact with computable traceability in a nice way:

\begin{theorem} For each $k$ there is a $k$-surviving, not $(k+1)$-surviving degree which is computably traceable.
\end{theorem}

A more complicated picture emerges in Section $3$ when we consider covering $\omega^\omega$ with closed sets. Here a difficulty arises in the effective setting with no classical analogue: there is no computable way to pass from a tree which branches at most $n$ times at each branching node (``$n$-tree") to a tree which branches exactly $n$ times at each branching node (``$n$-branching tree"). This gives us two separate tracks of cardinal characteristics, and leads to a somewhat messy picture. We show, for example:

\begin{theorem} There is a computably traceable degree containing a globally branch surviving function that is not $3$-globally tree surviving.
\end{theorem}

(These notions are defined in the beginning of Section $3$.)

\begin{theorem}\label{acctreeforce} There is a real $A\in\omega^\omega$ which is not a path through any computable $k$-tree for any $k\in\omega$ but which does not compute any $f\in 3^\omega$ which is not a path through any computable $2$-branching tree.
\end{theorem}

(Following the results of this paper, the first author showed \cite{ongay} that the forcing used to prove this latter theorem also establishes a new possible separation on the set-theoretic side.)

\bigskip

We recommend Blass' paper \cite{blass} to the reader further interested in cardinal characteristics. We would like to thank Arnie Miller for providing useful feedback on an early draft of this paper.


\section{Surviving degrees}

\begin{definition} A {\em $k$-branching tree } is a tree such that every node has either $1$ or $k$ many successors.
\end{definition}

The following definition is due to  \cite{newelski1993ideal}:

\begin{definition}
The {\em $k$-localization number}, $\mathfrak{L}_{k}$, is the smallest cardinal such that $(k+1)^{\omega}$ can be covered by $k$-branching trees.
\end{definition}

Following Rupprecht's analogy, the computability-theoretic version of the localization number is the following highness property:

\begin{definition} A function $f\in (k+1)^\omega$ is {\em $k$-surviving } if it is not a path through any computable $k$-branching subtree of $(k+1)^{<\omega}$. We say that a Turing degree is $k$-surviving if it computes a $k$-surviving function.
\end{definition}

We call these $k$-surviving degrees since they are the ones that go into the forest of $k$-branching trees and are able to escape it: they survive the experience. Note that this definition requires $k>0$ to make sense, and for $k=1$ trivializes: ``$1$-surviving" is just ``non-computable." So we are only interested in $k\ge 2$.

Before we begin analyzing the $k$-surviving degrees, there is a subtlety here which will matter later:

\begin{definition} A {\em $k$-tree } is a tree such that every node has at least $1$ and at most $k$ successors.
\end{definition}

Classically, the localization numbers can be equivalently defined in terms of $k$-trees rather than $k$-branching trees, and on the computability-theoretic side we can effectively pass from a $k$-tree contained in $n^{<\omega}$ for finite $n$ to a $k$-branching tree containing it, so this is not an issue at the moment. However, in general a computable $k$-tree merely contained in $\omega^{<\omega}$ may not be contained in a computable $k$-branching tree, as we will see below; this will give us two distinct computable analogues of the class of localization numbers in $\omega^{<\omega}$ in section $3$.

As an initial observation, it is easy to see that the $k$-surviving degrees form a hierarchy as $k$ varies:

\begin{lemma}\label{surviving implications}
Given $k\geq s\geq 2$, a $k$-surviving degree is also an $s$-surviving degree. 
\end{lemma}

\begin{proof} Fix a surjection $g:k+1\rightarrow s+1$, and let $g^{\ast}: (k+1)^{<\omega}\rightarrow (s+1)^{<\omega}$ and $\widehat{g}:(k+1)^{\omega}\rightarrow (s+1)^{\omega}$ be the induced computable surjections on the corresponding sets of finite or infinite strings. If $T\subseteq(s+1)^{<\omega}$ is a computable $s$-tree, then $(g^{\ast})^{-1}[T]\subseteq (k+1)^{<\omega}$ is a computable $k$-tree. This means that if $A\in (k+1)^{\omega}$ is a $k$-surviving function then $\widehat{g}(A)\in (s+1)^{\omega}$ is an $s$-surviving function: if $\widehat{g}(A)$ were in some computable $s$-tree, then pushing this forward we would have a computable $k$-tree containing $A$.



\end{proof}

Newelski and Roslanowski showed in \cite{newelski1993ideal} the following: $(i)$ for $k\geq 2$, $\mathfrak{L}_{k}\geq \max\{cov(\mathcal{M}), cov(\mathcal{N})\}$, $(ii)$ $\mathfrak{L}_{k+1}\leq\mathfrak{L}_{k}$, and $(iii)$ it is consistent that $\mathfrak{L}_{k+1}<\mathfrak{L}_{k}$. We can mimic all of those results in the computable side. Lemma \ref{surviving implications} is an analogue for $\mathfrak{L}_{k+1}\leq\mathfrak{L}_{k}$. All the subsets of $n^{\omega}$ that are covered by a $k$-branching computable tree (with $k<n$) are effectively meager and null, giving us:

\begin{theorem}\label{surviving null}
All the degrees that compute a Schnorr random real are $k$-surviving for all $k\geq 2$. In particular, there is a $k$-surviving degree that is DNC and a $k$-surviving degree that is not computable traceable.
\end{theorem}

\begin{theorem}\label{surviving meager}
All the degrees that compute a weak 1-generic (equivalently, all hyperimmune degrees) are $k$-surviving for all $k\geq 2$. In particular, there is a $k$-surviving degree that is weakly Schnorr engulfing.
\end{theorem}

Since there is a Schnorr random which is hyperimmune-free (see e.g. \cite{bbtnn} \S 4.2 (2))  and a hyperimmune degree that does not compute a Schnorr random (See e.g. \cite{niesbook} Theorem 1.8.37) we have:

\begin{corollary}
For all $k\geq 2$ there is a $k$-surviving degree which does not compute a Schnorr random.
\end{corollary}

\begin{corollary}
For all $k\geq 2$ there is a $k$-surviving degree which is hyperimmune free.
\end{corollary}

Notice that Theorem \ref{surviving null} and Theorem \ref{surviving meager} are the effective analogues of $\mathfrak{L}_{k}\geq \max\{cov(\mathcal{M})$ and $cov(\mathcal{N})\}$ respectively.

Finally, we turn to the converse of Lemma \ref{surviving implications}: we show that the $k$-surviving degree are truly hierarchical, i.e., that there are $k$-surviving degrees that are not $s$-surviving, with $s>k$. This is the computable analogue of the conditional consistency of $\mathfrak{L}_{k+1}<\mathfrak{L}_{k}$.

\begin{theorem}\label{surviving hierarchy}
Given $k\geq 2$, there is a $k$-surviving degree that is not $\ell$-surviving for $\ell\geq k+1$. Furthermore, it is possible to make this degree computable traceable. 
\end{theorem}

\begin{proof}

We will force with computable trees, $T$, in $(k+1)^{<\omega}$ such that for every $s\in T$ there is $t$ extending $s$ such that $t$ has more than one successor and, for all $s\in T$, we have that $|\{t\in T: s<t \ \& |t|=|s|+1\}|$ is either $1$ or $k+1$.

We will construct a function $A:\omega\rightarrow k+1$ that satisfy two types of requirements:\begin{itemize}

\item $R_{e}$: Given the $e$-th computable $k$-tree in $(k+1)^{<\omega}$, $A$ is not one of its branches.

\item $P_{e}$: $\varphi_{e}^{A}$ is either partial or is a branch of a computable $k+1$-tree of $\omega^{<\omega}$. In particular, if $\varphi_{e}^{A}:\omega\rightarrow \ell+1$ with $\ell\geq k+1$ then $\varphi_{e}^{A}$ is not $\ell$-surviving.
\end{itemize}

We begin at stage $s=0$ by setting $T_{0}=(k+1)^{<\omega}$. Now suppose that at stage $s+1$ we have a computable $k+1$-branching tree $T_{s}$ such that every branch satisfies all requirements $R_{j}$ and $P_{j}$ for $j<s$.

To satisfy $R_{s}$ we just need to extend the current stem (or root), $r_s$, in such a way that it is not longer in the $s$-th $k$-tree in $(k+1)^{<\omega}$. This is possible because $T_{s}$ has $k+1$ options every time it branches. 
Satisfying $P_{s}$ is more complicated, however, and we have two cases which need to be handled separately. The first case happens if there is $t\in T_{s}$  and $n$ such that $r_{s}\subseteq t$ and  $\varphi_{s}^{t'}(n)$ diverges for all $t'\supset t$. Setting $T_{s+1}$ to be the subtree of $T_{s}$ consisting of nodes comparable with $t$ then trivially satisfies $P_s$.


Now suppose we are unable to force partiality in this way. For all $t\in T_{s}$ extending $r_{s}$ and $n\in \omega$ there is a $t'\in T_{s}$ extending $t$ such that $\varphi_{s}^{t'}(n)$ converges. Now we will create $T_{s+1}$ in such a way that for all branches $A$ of $T_{s+1}$ we have that $\varphi_{s}^{A}$ is total. Furthermore, we can find a computable $k+1$-tree, $U_{s}$, such that for all branches of $T_{s+1}$, $\varphi_{s}^{A}$ is a branch of $U_{s}$. For convenience we will describe $T_{s+1}$ as a function  $(k+1)^{<\omega}\rightarrow (k+1)^{<\omega}$, denote by $t_{\sigma}$, and $U_{s}$ as function $(k+1)^{<\omega}\rightarrow \omega^{<\omega}$, denote by $u_{\sigma}$.

Before starting the construction we can assume one more hypothesis: for all $\tau\in T_{s}$ there exist $\tau_{0},..., \tau_{k}$ in $T_{s}$ extending $\tau$ and $n\in \omega$ such that $\varphi_{s}^{\tau_{i}}(\ell)$ converges for all $i<k+1$ and all $\ell<n$. Also, we need that $\varphi_{s}^{\tau_{i}}\restrict n\neq \varphi_{s}^{\tau_{j}}\restrict n$ for all $i\neq j<k+1$.\footnote{The use of $\tau$ in this paragraph instead of $t$ will simplify the reading later.}

If there is a $\tau$  extending $r_{s}$ such that the above hypothesis is false, that means that $\varphi_{s}^{A}$ can have at most $k$ different values as long as $\tau$ is an initial segment of $A$. Therefore, defining $T_{s+1}$ the subtree of $T_{s}$ extending $\tau$, we can find a computable tree $U_{s}$ with at most $k$ branches such that $\varphi_{s}^{A}$, with $A\in [T_{s+1}]$, is always one of those branches.\footnote{Here $[T]$ is all the branches of the tree $T$.}

Now, back to the construction, our strategy will be define for each node: first find an extension that splits; then, look for extensions of each node in the split (there are exactly $k+1$ of them) that make the function $\varphi_{e}^{t}$ different to each other, with that we keep the $k+1$ branching and we can use the information to define $U_{s}$.

Bringing the strategy to work, at the first stage, let $r_{s}=t_{\emptyset}$. Then look for for the first split above $t_{\emptyset}$ and call those nodes $\tau_{0},..., \tau_{k}$. Next, look for $t_{0},..., t_{k}$ extending  $\tau_{0},..., \tau_{0}$ respectively and $n_{\emptyset}\in \omega$ with $0<n_{\emptyset}$ such that $\varphi_{s}^{t_{i}}\restrict n_{\emptyset}\neq \varphi_{s}^{t_{j}}\restrict n_{\emptyset}$ for all $i\neq j<k+1$ and $\varphi_{s}^{t_{i}}(\ell)$ converges for all $i<k+1$ and all $\ell<n_{\emptyset}$. Define $u_{\emptyset}=\emptyset$ and $u_{i}=\varphi_{s}^{t_{i}}\restrict_{n_{\emptyset}}$.

In general, given $t_{\sigma}$ with $\sigma\in (k+1)^{<\omega}$, look for the first split above $t_{\sigma}$ and call those nodes $\tau_{\sigma 0},..., \tau_{\sigma k}$. Next, look for $t_{\sigma 0},..., t_{\sigma k}$ extending  $\tau_{\sigma 0},..., \tau_{\sigma k}$ respectively and an $n_{\sigma}\in \omega$ with $|\sigma|<n_{\sigma}$ such that $\varphi_{s}^{t_{\sigma i}}\restrict n_{\sigma}\neq \varphi_{s}^{t_{\sigma j}}\restrict n_{\sigma}$ for all $i\neq j<k+1$ and $\varphi_{s}^{t_{\sigma i}}(\ell)$ converges for all $i<k+1$ and all $\ell<n_{\sigma}$. Define $u_{\sigma i}=\varphi_{s}^{t_{\sigma i}}\restrict_{n_{\sigma}}$.

Since $T_{s}$ is computable we have that both $T_{s+1}$ and $U_{s}$ are computable. Furthermore, each split in $U_{s}$ is at most of size $k+1$ so it is a $k+1$-tree and, by construction, given a branch $A$ of $T_{s+1}$ we have that $\varphi_{s}^{A}$ is a branch of $U_{s}$.

Furthermore, if $A\in [T_{s+1}]$ and $\varphi_{s}^{A}$ is total then we can define a computable trace $\phi_{s}:\omega\rightarrow [\omega]^{<\omega}$ such that $\phi_{s}(n)$ is the $n$-th level of $U_{s}$. Notice that $|\phi_{s}(n)|\leq (k+1)^{n}$ and that for all branches $A$ of $T_{s+1}$ we have that $\varphi_{s}^{A}$ goes through $\phi_{s}$.

Finally, $A\in \bigcap_{s\in \omega} [T_{s}]$ is a $k$-surviving degree that is computably traceable degree and not $k+1$-surviving (or $\ell$-surviving for $\ell\geq k+1$).

\end{proof}

\begin{corollary}
Given $k\geq 2$, there is a $k$-surviving degree that is not DNC and not weakly Schnorr Engulfing.
\end{corollary}

Notice that the above construction for $k=1$ give us a non-computable set that is not a $2$-surviving degree.

To further pin-point the location of the surviving degrees in the effective Chicho\'n Diagram it is necessary to compare them to the DNC degrees. This question is still open:

\begin{question}
Is there a DNC degree that is not $k$-surviving?
\end{question}

In the same spirit, we may also ask:

\begin{question}
Is it possible to make a $k$-surviving degree that is not $k+1$ surviving and not computable traceable?
\end{question}


\section{Globally surviving degrees}

As seen in the proof of theorem \ref{surviving hierarchy}, it is possible to have degrees such that all their functions $f:\omega\rightarrow \omega$ go through a $k$-subtree of $\omega^{<\omega}$. Therefore, we can define degrees that survive $k$-trees in $\omega^{\omega}$. Here, however, we run into the subtlety mentioned earlier: that since we are no longer working with trees over a finite set, ``computable $k$-branching tree" and ``computable $k$-tree" may behave differently. This leads to two separate tracks of highness notions:

\begin{definition}
\begin{enumerate}
\item A function $g:\omega\rightarrow\omega$ is {\em $k$-globally branch surviving } if it is not a path through any computable $k$-branching tree; a Turing degree $B$ is $k$-globally branch surviving if it computes a $k$-globally branch surviving function.


\item A function $g:\omega\rightarrow\omega$ is {\em globally branch surviving } if it is $k$-globally branch surviving for every $k\in\omega$; a Turing degree $B$ is globally branch surviving if it computes a globally branch surviving function.


\item A function $g:\omega\rightarrow\omega$ is {\em $k$-globally tree surviving } if it is not a path through any computable tree; a Turing degree $B$ is $k$-globally tree surviving if it computes a $k$-globally tree surviving function.

\item A function $g:\omega\rightarrow\omega$ is {\em globally tree surviving } if it is not a path through any computable $k$-tree for any $k\in\omega$; a Turing degree $B$ is globally tree surviving if it computes a globally tree surviving function.

\end{enumerate}
\end{definition}

The distinction between $k$-branching trees and $k$-trees is significant. Trivially a $k$-globally tree surviving degree is also $k$-globally branch surviving, and similarly a globally tree surviving degree is globally branch surviving. However, no other coarse implication exists.

We begin by showing that global branch and global tree survival differ wildly on the level of individual functions:




\begin{proposition}\label{Tree not cover by branch}
There is a globally-branch surviving function that is not $3$-globally tree surviving. 
\end{proposition}

\begin{proof} We will define a computable $3$-tree of $\omega^{<\omega}$ which is not covered by any computable $k$-branching tree. The right most path of this $3$-tree will be globally-branch surviving but, since it is a branch of a $3$-tree, it is not $3$-globally tree surviving.

We will define $T$ by stages, beginning with $T_0=\emptyset$ and obeying the following rules:\begin{enumerate}

\item If $p\in T_{s}$ then $p0\in T_{s+1}$.

\item If $p\in T_{s}$, then for each $i<s+1$ we will decide whether $p^{\frown}i=pi\in T$ at stage $s+1$.

\item We will fix a computable permutation that maps $\omega$ with all the pairs $\langle e, k\rangle$ with $k>2$. If $p\in T$, $|p|=n=\langle e, k\rangle$, the successors of $p$ will deal with $\varphi_{e}$ as it were a $k$-branching tree. In particular, if $\varphi_{e}$ is a $k$-branching tree then the right most successor of $p$ is not in it.
\end{enumerate}

(Strictly speaking, at a given stage we have a tree together with a finite set of forbidden nodes, but for simplicity we speak of just building a tree and making declarations.)

Assume that at stage $s$ we have $p\in T_{s}$, $\vert p\vert=n=\langle e, k\rangle$. There are now three cases:

\textbf{Case 1} No successor of $p$ is in $T_{s}$ other than perhaps $p0$.

If $\varphi_{e,s}$ does not look like a $k$-branching tree containing $p$, then we set $pi\not\in T$ for all $0<i\leq s$.
Otherwise, we check whether $p0\in \varphi_{e,s}$; if it is not, we declare that $ps\notin T$.

If $p0\in \varphi_{e,s}$, we further check if $ps\in \varphi_{e,s}$. If it is then we declare $ps\notin T$; if $ps\notin \varphi_{e,s}$ then we declare that $ps\in T_{s+1}$.

\textbf{Case 2} There are exactly two successors of $p$ in $T_{s}$ (one of which is $p0$).

Being in this case means that, at some stage, $\varphi_{e}$ looked like a $k$-branching tree and that $p0\in \varphi_{e,s}$. If there are not exactly $k$-many successors of $p$ in $\varphi_{e,s}$ then we declare that $ps\notin T$; otherwise, we put $ps\in T_{s+1}$. Note that $ps\not\in \varphi_{e, s}$ by use constraints, so we are free to make this decision at this time.


\textbf{Case 3} If there are three successors of $p$ in $T_{s}$ we declare $ps\notin T$.

The union $\bigcup_{s\in\omega} T_s$ is a computable $3$-tree whose right-most path  is not in any computable $k$-branching tree, and this finishes the proof.
\end{proof}

At a first glance, it may appear that the rightmost branch through the tree constructed above has significant computational power. Interestingly, this is not quite true:


\begin{theorem}\label{branch treaceable}
There is computably traceable degree that computes a globally branch surviving function that is not $3$-globally tree surviving.
\end{theorem}

\begin{proof}

We will use the same tree as in the proof of Proposition \ref{Tree not cover by branch}, with a slight modification. Having as above an effective enumeration $\{\langle e_i, k_i\rangle: i>0\}$ of $\omega\times\omega_{\ge 3}$ --- note that we do not include a $0$th term, for notational convenience below --- we will fix a computable sequence in which every natural number $>0$ occurs infinitely often; we use $$1,1,2,1,2,3,1, 2,3, 4, ...$$ The terms in this sequence tell us what levels of our tree will deal with each pair in $\omega\times\omega_{\ge 3}$. For example, $\langle e_1, k_1\rangle$ will be dealt with at levels $0, 1, 3, ...$ and $\langle e_2, k_2\rangle$ at levels $2, 4, 7, ...$.


We build a sequence of triples $\langle p, T, g\rangle$ with the following properties:\begin{itemize}
\item $p\in \omega^{<\omega}$.
\item $T$ is a computable $3$-tree that is not covered by any computable $3$-branching tree, with $p\in T$ (indeed we may assume that $p$ is the stem of $T$).
\item $g:T\rightarrow \omega$ is such that: for all $\sigma \in T, n\in \omega$, there is an extension $\tau\in T$ of $\sigma$ with $g(\tau)=n$. 
\item Finally, if $\varphi_{e}$ is a $k$-branching tree, then there are infinitely many nodes in $T$ with a successor not in $\varphi_e$.
\end{itemize}

Here, $p$ is the initial segment of the function we are constructing, $T$ is the tree of possible future extensions (so the real we produce is a branch of $T$), and if $g(\sigma)>0$ then the successors of $\sigma$ deal with $\varphi_{e_{g(\sigma)}}$ as if it were a $k_{g(\sigma)}$-tree (where $\{\langle e_i, k_i\rangle: i>0\}$ is as above) --- that is, the labelling function $g$ assigns tasks to each node of the tree.


To start, fix (noneffectively) an enumeration of all computable $3$-branching trees. During the construction of $A$ we want to satisfy two families of requirements:\begin{itemize}

\item $R_{e}$: If $\varphi_{e}$ is a $k$-branching tree for any $k\in \omega$ then $A$ is not a branch of it. This will make $A$ a globally branch surviving degree.

\item $P_{e}$: If $\varphi_{e}^{A}$ is total then it goes through some computable trace bounded by $f(n)=3^{n}$. This will make $A$ computably traceable.
\end{itemize}

We will start our construction with $\langle p_{0}, T_{0}, g_{0}\rangle=\langle \emptyset, T, g_{0}\rangle$ with $T$ as described in the first paragraph and $g_{0}$ being constant at every level, and mapping nodes on the $i$th level to the $i$th term of the sequence $$\langle 1, 1, 2, 1, 2, 3, 1, 2, 3, 4, ....\rangle.$$ Our construction breaks into even and odd stages, handling the $R$- and $P$-requirements respectively. The former are easily satisfied, while the latter require a construction.

{\bf Even stages.} At stage $s=2e$ we have $\langle p_{s}, T_{s}, g_{s}\rangle$. Given the $e$-th computable $3$-branching tree, we look for an extension of $p_{s}$ that avoids it. To do this, we look for $r$ such that the $e$-th branching tree is a $k_{r}$ branching tree and is describe by $\varphi_{e_{r}}$ (we are using $\langle e_{r}, k_{r}\rangle$). Now, we look for an extension of $p_{s}$ in $T_{s}$, called it $\tau$, such that $g_{s}(\tau)=r$, we focus on the successor of $\tau$ that avoids $\varphi_{e_{r}}$. That successor will be $p_{s+1}$. We define $T_{s+1}$ to be the subtree of $T_{s}$ that extends $p_{s+1}$ and we let $g_{s+1}$ to be $0$ for all the initial segments of $p_{s+1}$ and be the same as $g_{s}$ for the other members of $T_{s+1}$.

{\bf Odd stages.} At $s=2e+1$ we have $\langle p_{s}, T_{s}, g_{s}\rangle$. If there is $\tau\in T_{s}$ such that $\varphi^{A}_{e}$ is not total for all branches $A$ of $T_{s}$ extending $\tau$ then we define $p_{s+1}=\tau$ and we make $T_{s+1}$ and $g_{s+1}$ as in stage $2e$. If there is not such an extension, we define $p_{s+1}$ to be the first node extending $p_{s}$ such that $g_{s}(p_{s+1})=1$. To define $T_{s+1}$, we want to prune $T_s$ in such a way that $\varphi^A_e$ is in a computable trace bounded by $3^n$ whenever $A$ is a branch, and $g_{s+1}$ should then be defined accordingly.

Specifically, under the assumption above we define $T_{s+1}$ and $g_{s+1}$ by the following steps:

\begin{enumerate}
\item At every stage $t$ there is at most one node entering $T_{s+1,t}$ that is a successor of a node with $g_{s+1}\neq 0$.

\item At every stage $t$, if we declare that $\sigma\in T_{s}$ will belong to $T_{s+1}$ and $\sigma$ is not the successor of a node with $g_{s+1}\neq 0$ then $\sigma\in T_{s+1,t+1}$.

\item We will fix an enumeration of $\omega^{<\omega}$, after adding the nodes of the rule above, we will give an opportunity to the nodes that are successors of a node with $g_{s+1}\neq 0$ by the order of the enumeration.  

\item At stage $t=0$, $p_{s+1}\in T_{s+1,0}$, $g_{s+1,0}(\sigma)=0$ for all $\sigma \prec p_{s+1}$ and $g_{s+1}(p_{s+1})=g_{s}(p_{s+1})=1$. Here we start the next stage.

\item  If $q$ just entered $T_{s+1,t}$ and it is a successor of a node with $g_{s+1}\neq 0$, we will look for $\tau_{\sigma}\in T_{s}$ extending $\sigma$, for each leaf of $T_{s+1,t}$, and $\tau\in T_{s}$ extending $q$ such that there is $n,m\in \omega$ with $m>|q|$, $n<m$ and $\varphi_{e}^{\tau}(n)\neq \varphi_{e}^{\tau_{\sigma}}(n)$: Also, we want $\varphi_{e}^{\tau}$ and $\varphi_{e}^{\tau_{\sigma}}$ to converge for the whole interval $[0,m]$. Finally, $g_{s}(\tau_{\sigma})=g_{s+1}(\sigma)$.

\item For every leaf $\sigma$ of $T_{s+1, t}$ we declare that the nodes between $\sigma$ and $\tau_{\sigma}$ will be in $T_{s+1}$, that there is no split between these nodes, and that $g_{s+1}(\rho)=0$ for such a node. We also change the value of $g_{s+1}(\sigma)$ to $0$ and set $g_{s+1}(\tau_{\sigma})=g_{s}(\tau_{\sigma})$.

\item Given $\tau$ from step 5 (the one extending $q$), we look at the subsequence $\langle g_{s+1} (q\restrict i_{m})\rangle$ made by all the nonzero values of $\langle g_{s+1} (q\restrict i) :i<|q| \rangle$  and we look for an extension of $\tau$, called it $q'$, in $T_{s}$ such that $\langle g_{s+1} (q\restrict i_{m})\rangle\frown g_{s}(q')$ is an initial segment of $\langle 1, 1,2,1,2,3,... \rangle$.

\item Given $q$ and $q'$ as the rules above, we declare that $q'\in T_{s+1, t}$ as well as all its successors in $T_{s}$ and initial segments; we also declare that $g_{s+1}(q')=g_{s}(q')\neq 0$; that there are no splits in $T_{s+1}$ between $q$ and $q'$, and that $g_{s+1}(\sigma)=0$ for all $q\preceq \sigma \prec q'$. Now we start the next stage.
\end{enumerate}

This construction produces a triple $\langle p_{s+1}, T_{s+1}, g_{s+1}\rangle$ with the desired form. Moreover, the function $g_{s+1}$ changes value from $g_s$ at each node at most once and will not change once a successor of the node enters $T_{s+1}$. Since all nodes of $T_{s+1}$ have a successor, $g_{s+1}$ is computable. Furthermore, by step $5$ we know that $\varphi_e^A$ is total whenever $A$ is a branch of $T_{s+1}$. So to complete the proof we just need to check that there is a single computable trace capturing all this functions. This is provided by $$U_{e}=\{\varphi_{e}^{\sigma}\restrict n : n\in \omega, \sigma\in T_{s+1}\}.$$ That is, we claim that the $n$th level of $U_e$ has size at most $3^n$. This is because $T_{s+1}$ is a $3$-tree, so the only way this could fail would be if there were splitting nodes in $T_{s+1}$ which saw no new convergence of $\varphi_e$ (since then by waiting for more splittings further up the tree, we could produce more than $3^n$ values of the computation and hence more than $3^n$-many nodes on $U_e$ of height $n$). However, by construction we generate new values of $\varphi_e^\sigma$ exactly when we split, so this cannot happen.


So the degree of the function $A=\bigcup_{n\in \omega}p_{n}$ is as we desired. 

\end{proof}

\begin{question}\label{branch vs tree}
Is there a globally-branch surviving degree that is not $3$-globally tree surviving? 
\end{question}

\begin{question}
Is it true that $k,s\geq 3$ there is a $k$-globally branch surviving degree that is not $s$-globally tree surviving?
\end{question}

Clearly, a globally-branch surviving degree is a $k$-globally branch surviving degree. Also, a globally-tree surviving degree is a $k$-globally tree surviving degree. To really show that this degrees make a hierarchy, we need to make the following observations.

\begin{lemma}\label{tree hierarchy}
Given $k\geq s\geq 2$, a $k$-globally tree surviving degree is also an $s$-globally tree surviving degree. 
\end{lemma}

\begin{proof}
By definition, an $s$-tree is also a $k$-tree, so, if you survive all $k$-trees, in particular, you survive all $s$ trees.
\end{proof}

\begin{lemma}\label{branch hierarchy}
Given $k\geq s\geq 2$, a $k$-globally branch surviving degree is also an $s$-globally branch surviving degree. 
\end{lemma}

\begin{proof}
Given $k\geq s$ notice that if $A$ is not an $s$-globally surviving degree then all the total functions that it computes are the branch of a computable $s$-branching tree. Now, notice that we can make a computable $s$-branching tree of $\omega^{\omega}$ into a $k$-branching tree in a uniform way, so all the total functions that $A$ computes are a branch of a $k$-branching tree.

This shows that $A$ is not $k$-globally surviving.
\end{proof}


Note that contrary to what the name suggests, being $k$-globally surviving is weaker than being $k$-surviving: a function that is $k$-surviving is also $k$-globally surviving since $(k+1)^{<\omega}\subseteq (\omega)^{<\omega}$ and the fact that if $T\subseteq (\omega)^{<\omega}$ is a $k$-branching tree then $T\cap (k+1)^{<\omega}$ is a $k$-tree. We have:

\begin{lemma}
If $A$ is a $k$-surviving degree then it is a $k$-globally surviving degree.
\end{lemma}

The converse fails badly, however:

\begin{theorem}\label{globally surviving vs surviving}
There is a Turing degree $A$ that is globally tree surviving but not $2$-surviving.
\end{theorem}

\begin{proof}
As before, we will do forcing with stems and trees, but this time we will use accelerating subtrees of $\omega^{<\omega}$.

\begin{definition}\label{acctreecompdef} An {\em accelerating tree } is a subtree $T\subseteq\omega^{<\omega}$ such that if $\sigma\in T$ is a splitting node with $n$ splitting initial segments, then $\sigma$ has more than $n+2$ immediate successors.
\end{definition}

We will force with conditions of the form $\langle p, T\rangle$ were $p\in \omega^{<\omega}$ and $T$ is a computable accelerating subtree of $\omega^{<\omega}$ extending $p$.

We will construct $A\in \omega^{<\omega}$ with the following two requirements:\begin{itemize}

\item $R_{e,k}$: $A$ is not a branch of the $e$-th computable $k$-subtree of $\omega^{<\omega}$. This will make $A$ a globally tree surviving degree.

\item $P_{e}$: $\varphi_{e}^{A}$ is either not total or there is $n$ such that $\varphi_{e}^{A}(n)\geq 3$ or there is a computable $2$-branching tree of $3^{<\omega}$ such that $\varphi_{e}^{A}$ is a branch of it. This will make $A$ not a $2$-surviving degree.
\end{itemize}

We will set $\langle p_{0}, T_{0}\rangle=\langle \emptyset, \omega^{<\omega}\rangle$.

At stage $s=2\langle e, k\rangle$, if $U_{e}$ is the $e$-th computable $k$-branching subtree of $\omega^{<\omega}$ then we look for an extension of $p_{s}$ that has at least $k+1$ successors and we define $p_{s+1}$ to be the successor that is not in $U_{e}$. We define $T_{s+1}$ to be the subtree of $T_{s}$ extending $p_{s+1}$. Since $T_{s}$ is computable, given $p_{s+1}$, $T_{s+1}$ is computable. 

At stage $s=2e+1$ we have four cases:

\textbf{Case 1} If there is $\sigma\in T_{s}$ extending $p_{s}$ such that $\varphi_{e}^{\sigma}$ is not total, then let $p_{s+1}=\sigma$ and define $T_{s+1}$ to be the subtree of $T_{s}$ extending $p_{s+1}$. Here we satisfy $P_{e}$ by avoiding totality.

\textbf{Case 2} If there is $\sigma\in T_{s}$ extending $p_{s}$ and $n\in \omega$ such that $\varphi_{e}^{\tau}(n)\geq 3$, then let $p_{s+1}=\sigma$ and define $T_{s+1}$ to be the subtree of $T_{s}$ extending $p_{s+1}$. This satisfy $P_{e}$.

\textbf{Case 3} If there is $\sigma\in T_{s}$ extending $p_{s}$ such that there a no $\tau_{1},\tau_{2}\in T_{s}$ extending $\sigma$ such that $\varphi_{e}^{\tau_{1}}\neq \varphi_{e}^{\tau_{2}}$. In this case we will satisfy $P_{e}$ by the fact that $\varphi_{e}^{A}$ with $A$ extending $\sigma$ is computable if it is total.

\textbf{Case 4} For this case, we need that the other three cases are not happening. We will define $p_{s+1}=p_{s}$ and we will prune $T_{s}$.

We will define this prune by levels, here nodes at level $n$ will have exactly $n$ splits before them. Furthermore, during the prune we will define a $U_{e}$ a $2$-tree (remember that a $2$-tree and a $2$-branching tree are the same) such that for all the branches $A$ of $T_{s+1}$, $\varphi_{e}^{A}$ is a branch of $U_{e}$.

At level $0$ we will have a unique node: $p_{s+1}$. Furthermore, we will add $\emptyset$ to $U_{e}$.

Now, assume that $\tau\in T_{s}$ is a node in level $n-1$ with $n\geq 1$. We know that there are exactly $n-1$ splits before $\tau$ and that, in order to make $T_{s+1}$ accelerating, the next split should have at least $n$ nodes.

To do this we will take $\sigma\in T_{s}$ extending $\tau$ that has at least $3^{n}$ successors $\tau_{0}^{0},...,\tau_{3^{n}-1}^{0}$ in $T_{s}$. We will look for $m\in \omega$, and $\sigma_{i}\in T_{s}$ extending $\tau_{i}^{0}$ such that $\varphi_{e}^{\sigma_{i}}(t)\downarrow$ for $t<m$ and that there are $i_{0}, j_{0}<3^{n}$ such that $\varphi_{e}^{\sigma_{i_{0}}}\restrict m\neq \varphi_{e}^{\sigma_{j_{0}}}\restrict m$.

Let $m_{0}<m$ be the minimal number such that there are $i.j<3^{n}$ with $\varphi_{e}^{\sigma_{i}}(m_{0})\neq \varphi_{e}^{\sigma_{j}}(m_{0})$. Since we know that  $\varphi_{e}^{\sigma_{i}}(m_{0})<3$, we have that there is $k_{0}<3$ with at least $3^{n-1}$ $\sigma_{i}$ such that $\varphi_{e}^{\sigma_{i}}(m_{0})=k_{0}$.

We will define $\tau'_{0}$ to be one of the $\sigma_{i}$ such that $\varphi_{e}^{\tau'_{0}}(m_{0})\neq k_{0}$ and we will define $\tau_{0}^{1},..., \tau_{3^{n-1}}^{1}$ to be $3^{n-1}$ of the $\sigma_{i}$ with $\varphi_{e}^{\sigma_{i}}(m_{0})=k_{0}$.

To clarify, at this moment we have $\tau'_{0}\in T_{s}$ (that is a candidate to be a member of the n-th level) and $3^{n-1}$ nodes of $T_{s}$, $\tau_{0}^{1},..., \tau_{3^{n-1}}^{1}$, such that for $i,j<3^{n-1}$ $ \varphi_{e}^{\tau_{i}^{1}}\restrict (m_{0}+1)= \varphi_{e}^{\tau_{j}^{1}}\restrict (m_{0}+1)$, $\varphi_{e}^{\tau_{i}^{1}}(m_{0})\neq \varphi_{e}^{p'_{0}}(m_{0})$ but $\varphi_{e}^{\tau_{i}^{1}}\restrict m_{0}= \varphi_{e}^{\tau'_{0}}\restrict m_{0}$.

We can repeat the process to get $m_{1}\in \omega$, $m_{1}>m_{0}$, $\tau'_{1}\in T_{s}$ and $\tau_{0}^{2},..., \tau_{3^{n-2}}^{2}\in T_{s}$ all of them extending one of the nodes $\tau_{j}^{1}$ and have similar properties as the above paragraph. In other words: for $i,j<3^{n-2}$ $ \varphi_{e}^{\tau_{i}^{2}}\restrict (m_{1}+1)= \varphi_{e}^{\tau_{j}^{2}}\restrict (m_{1}+1)$, $\varphi_{e}^{\tau_{i}^{2}}(m_{1})\neq \varphi_{e}^{p'_{1}}(m_{1})$ but $\varphi_{e}^{\tau_{i}^{2}}\restrict m_{1}= \varphi_{e}^{\tau'_{1}}\restrict m_{1}$.

Repeating this process $n$ times, we get $m_{0}<...< m_{n-1}\in \omega$, $\tau'_{0}, ..., \tau'_{n-1}\in T_{s}$ all of them extending $p$ (even more specifically, they all come from a single split above $\tau$) such that $\varphi_{e}^{\tau_{i_{0}}}\restrict m_{j}=\varphi_{e}^{\tau_{i_{1}}}\restrict m_{j}$ for all $j\leq i_{0}, i_{1}< n$. In particular, if we find $\tau_{i}$ extending $\tau'_{i}$ such that $\varphi_{e}^{\tau_{i}}(t)\downarrow$ for $t<m_{n}+1$, we have that the tree created by $\varphi_{e}^{\tau_{i}}\restrict m_{n}$ is a $2$-tree. Notice that for all $i<n$, $\tau_{i}$ have exactly $n$ splits before it and the $n$-th split has $n$ successors. At this moment, we include $\tau_{i}$ at level $n$ and we include $\varphi_{e}^{\tau_{i}}\restrict m_{n}$ to $U_{e}$. There will be no more extension of $\tau$ in level $n$.

We now define $T_{s+1}$ to be the subtree generated by all the levels describe above. $T_{s+1}$ is an accelerating tree and for all the branches $A$ of $T_{s+1}$, $\varphi_{e}^{A}$ is a branch of $U_{e}$. This satisfy the requirement $P_{e}$.

The above argument shows that the Turing degree of $A=\bigcup_{n\in\omega} p_n$ is globally tree surviving but not $2$-surviving, so we are done.
\end{proof}

At the same time, the two notions are not too far apart in some sense. The proof above also gives the following result:

\begin{theorem}
Given $k\geq 2$, there is a $k$-globally surviving degree that is not $\ell$-globally surviving for $\ell\geq k+1$. Also, this degree is computable traceable. 
\end{theorem}

So despite the difference between the two, there is a parallel between survival and global survival (compare with Theorem \ref{surviving hierarchy}). 

\section{Further questions}

We end by mentioning three general directions for further work:

\bigskip

First, although we have established a number of facts about the effective localization numbers (or rather, the highness properties corresponding to localization numbers), there are still important questions left open with regards to their interactions with better-understood highness properties. For example, we showed that computably traceable degrees could witness separations between levels of the global survival hierarchy, but the general role of computable traceability here is open. We do not know whether there are non-computably traceable degrees which are {\em not } $k$-surviving for $k\ge 2$, or whether ``$k$-surviving for every $k$" implies not computably traceable. Similarly we can ask about the role of DNC in these contexts. (We have already mentioned above questions regarding the separations {\em within } the localization hierarchies.)

Second, returning to the relationship between local and global survivability notions, it seems likely that the combinatorial arguments of Theorems 2.9 and 3.10 used to establish that the respective degrees are not surviving can be modified to prove more than was necessary for those theorems --- namely that for every $k$, every $(k+1)$-branching subtree of $(k+2)^{<\omega}$ (not just element of $(k+2)^\omega$) computable from the generic added by either of the following two forcings is contained in some computable $(k+1)$-branching subtree of $(k+2)^{<\omega}$:\begin{itemize}
\item Forcing with computable accelerating trees with explicit stems.
\item Forcing with $k$-branching subtrees of $(k+1)^{<\omega}$ which have a splitting node above every node, again with explicit stems.
\end{itemize}
Appropriately relativized, these preservation results\footnote{These are paralleled and extended, on the set-theoretic side, by Lemma 3.12 of the first author's paper \cite{ongay}.} would show that forcing with the product of the forcings above produces a real which is not $(k+1)$-surviving; meanwhile, by the previous results the real produced {\em is } both globally branch surviving and $k$-surviving. This would separate the global and local survivability notions in a strong way.

Finally, while normally effective cardinal characteristic results build on results and techniques on the set-theoretic side, there is no reason the converse cannot happen. The accelerating tree forcing introduced in the proof of Theorem \ref{globally surviving vs surviving} in particular provides an example of this: following the work of this paper, in \cite{ongay} the first author showed that forcing with the $\omega_2$-length product with countable support of accelerating tree forcing produces a model giving a partial answer to a question of Blass in \cite{blass}. Localization numbers have not been studied as much in set theory as their more famous counterparts (the domination/bounding numbers, the reaping/splitting numbers, etc.), and so represent a possible area where effective methods may yield new set-theoretic results in general; even further analysis of accelerating tree forcing alone may be useful in this regard.



\bibliographystyle{abbrv}
\bibliography{biblio}

\end{document}